\documentclass[12pt,reqno]{amsart}

\usepackage{amsmath,amssymb,amsfonts,amsthm,mathrsfs,mathtools}
\usepackage[dvipsnames]{xcolor}
\usepackage{tikz}
\usepackage[margin=1in]{geometry}
\usepackage{enumitem}

\definecolor{linkred}{rgb}{0.78,0.03,0.08}
\definecolor{linkblue}{rgb}{0.0,0.2,0.67}
\definecolor{lime}{HTML}{A6CE39}

\usepackage[
  colorlinks=true,
  linkcolor=linkred,
  citecolor=linkblue,
  urlcolor=MidnightBlue,
  pagebackref,
  hypertexnames=false
]{hyperref}

\numberwithin{equation}{section}
\theoremstyle{plain}
\newtheorem{theorem}{Theorem}[section]

\newtheorem{lemma}[theorem]{Lemma}
\newtheorem{corollary}[theorem]{Corollary}

\theoremstyle{definition}
\newtheorem{remark}[theorem]{Remark}

\newcommand{\R}{\mathbb{R}}
\newcommand{\Sph}{\mathbb{S}}

\newcommand{\dx}{\,dx}
\newcommand{\dr}{\,dr}
\newcommand{\dt}{\,dt}
\newcommand{\ds}{\,ds}

\newcommand{\domega}{\,d\omega}

\newcommand{\abs}[1]{\left\lvert #1 \right\rvert}

\DeclareRobustCommand{\orcidicon}{%
  \begin{tikzpicture}
    \draw[lime, fill=lime] (0,0)
      circle [radius=0.16]
      node[white] {{\fontfamily{qag}\selectfont \tiny ID}};
    \draw[white, fill=white] (-0.0625,0.095) circle [radius=0.007];
  \end{tikzpicture}\hspace{-2mm}}
\foreach \x in {A, B}{%
  \expandafter\xdef\csname orcid\x\endcsname{%
    \noexpand\href{https://orcid.org/\csname orcidauthor\x\endcsname}{\noexpand\orcidicon}}}

\title[Critical Hardy--Rellich inequality]{An extension of a critical Hardy--Rellich inequality:
       explicit constants and the sharp weight range}

\author[M.~Majdoub]{Mohamed Majdoub\orcidB{}}

\address[M.~Majdoub]{Department of Mathematics, College of Science, Imam Abdulrahman Bin Faisal University, P.O.~Box 1982, Dammam, Saudi Arabia.\newline
Basic and Applied Scientific Research Center, Imam Abdulrahman Bin Faisal University, P.O.~Box 1982, 31441, Dammam, Saudi Arabia.}
\email{\tt mmajdoub@iau.edu.sa}
\email{\tt  med.majdoub@gmail.com}
\email{\tt mohamed.majdoub@fst.rnu.tn}


\begin{document}
\begin{abstract}
We revisit the critical Hardy--Rellich inequalities recently established by Castro~\cite{Castro2026}. Using the classical one-dimensional weighted Hardy inequality in Emden--Fowler variables, we prove that the weighted inequality holds exactly for $a\neq N$, with $a=N$ as the unique critical value. We derive explicit constants for $N\ge2$, obtain the sharp constant in dimension one, and extend the $\Delta u$ formulation to the Muckenhoupt range $-N<a<N(N-1)$, $a\neq N$. This provides a partial solution to an open problem posed in~\cite{Castro2026}.
\end{abstract}
\maketitle

%\tableofcontents

%==================================================================

\section{Introduction and main results}\label{sec:intro}

For $N\ge 1$ and $1\le p<N$, the classical Hardy inequality states that
\[
  \int_{\R^N}\frac{\abs{u}^{p}}{\abs{x}^{p}}\dx
  \le \left(\frac{p}{N-p}\right)^{p}\int_{\R^N}\abs{\nabla u}^{p}\dx
\]
for all $u\in C^\infty_c(\R^N)$; under the additional restriction
$u\in C^\infty_c(\R^N\setminus\{0\})$ the inequality extends to all $p\ne N$,
with the constant $\bigl(\tfrac{p}{\abs{N-p}}\bigr)^{p}$. The endpoint case
$p=N$ is well known to fail without a logarithmic correction.

The second--order analogue, the Rellich inequality, takes the form
\[
  \int_{\R^N}\frac{\abs{u}^{p}}{\abs{x}^{2p}}\dx
  \le \left[\frac{p^{2}}{N(N-2p)(p-1)}\right]^{p}
       \int_{\R^N}\abs{\Delta u}^{p}\dx
\]
for $N\ge 3$ and $1<p<N/2$ (see~\cite{Rellich,DaviesHinz}); the endpoint
$p=N=2$ likewise fails (\cite[\S 7]{Rellich2}).

Hardy--Rellich type inequalities, which place a first--order quantity on the
left and a second--order quantity on the right, were studied in
\cite{Costa, TertikasZographopoulos, Cazacu} for $1<p<N$. Again the case
$p=N$ is excluded, and for the same structural reason as above: the failure of
the first--order Hardy inequality at $p=N$ propagates to the second--order
estimate.

In~\cite{Castro2026} the author proves a striking critical--exponent
Hardy--Rellich inequality which covers the previously inaccessible case
$p=N$, at the cost of grouping the gradient and the lower--order term into a
single quantity $\nabla(u/\abs{x})$. The main results of~\cite{Castro2026}
read as follows.

\begin{theorem}[{\cite[Theorem 1.1]{Castro2026}}]\label{thm:castro11}
For every $N\ge 1$ there exists $C_N>0$ such that
\begin{equation}\label{eq:castro11}
  \int_{\R^N}\Bigl|\nabla\Bigl(\tfrac{u(x)}{\abs{x}}\Bigr)\Bigr|^{N}\dx
  \le C_N \int_{\R^N}\abs{\Delta u}^{N}\dx,
  \qquad u\in C^\infty_c(\R^N\setminus\{0\}).
\end{equation}
\end{theorem}

\begin{theorem}[{\cite[Theorem 1.2]{Castro2026}}]\label{thm:castro12}
Let $N\ge 1$ and $a<1$. There exists $C_{N,a}>0$ such that
\begin{equation}\label{eq:castro12}
  \int_{\R^N}\Bigl|\nabla\Bigl(\tfrac{u(x)}{\abs{x}}\Bigr)\Bigr|^{N}\abs{x}^{a}\dx
  \le C_{N,a}\int_{\R^N}\abs{D^{2}u}^{N}\abs{x}^{a}\dx,
  \qquad u\in C^\infty_c(\R^N\setminus\{0\}).
\end{equation}
\end{theorem}

\begin{corollary}[{\cite[Corollary 1.1]{Castro2026}}]\label{cor:castro}
For $-N<a<1$ there exists $C_{N,a}>0$ such that
\begin{equation}\label{eq:castrocor}
  \int_{\R^N}\Bigl|\nabla\Bigl(\tfrac{u(x)}{\abs{x}}\Bigr)\Bigr|^{N}\abs{x}^{a}\dx
  \le C_{N,a}\int_{\R^N}\abs{\Delta u}^{N}\abs{x}^{a}\dx,
  \qquad u\in C^\infty_c(\R^N\setminus\{0\}).
\end{equation}
\end{corollary}

One of the questions left open in~\cite{Castro2026}, namely Open Problem~1, concerns the determination of the optimal constants in~\eqref{eq:castro11} and~\eqref{eq:castrocor}.
\medskip

The constants in~\cite{Castro2026} are not explicit, and, as we shall see, the
range $a<1$ in~\eqref{eq:castro12} is far from optimal. Both limitations
originate in the use of an auxiliary boundedness result for a one--dimensional
integral operator, whose proof relies on Jensen's inequality. Jensen is sharp
only for $p=1$, and for $p>1$ it is strictly weaker than the classical weighted
Hardy inequality. The latter is available in any weighted Lebesgue exponent,
yields an explicit (and, for the one--dimensional problem, sharp) constant, and
covers a much wider range of weights.

Our contribution is to recast both the radial and the angular parts of the
argument as instances of a single one--dimensional weighted Hardy inequality in
the Emden--Fowler variable $t=\log r$. This both removes the artificial
restriction $a<1$ and produces explicit constants. The improvement is summarised
in the following statement.

\begin{theorem}\label{thm:main}
Let $N\ge 2$ and $a\ne N$. Then for every $u\in C^\infty_c(\R^N\setminus\{0\})$,
\begin{equation}\label{eq:main}
  \int_{\R^N}\Bigl|\nabla\Bigl(\tfrac{u(x)}{\abs{x}}\Bigr)\Bigr|^{N}\abs{x}^{a}\dx
  \le C_{N,a}\int_{\R^N}\abs{D^{2}u}^{N}\abs{x}^{a}\dx,
\end{equation}
and for $a\ne 0$ one may take the explicit constant
\begin{equation}\label{eq:const}
  C_{N,a} = 2^{N/2-1}\left[\left(\frac{N}{\abs{a}}\right)^{N}
                        + \left(\frac{N}{\abs{N-a}}\right)^{N}\right].
\end{equation}
The remaining value $a=0$ is covered by Theorem~\ref{thm:castro11}. For $N=1$,
the inequality holds for every $a\ne 1$ with the sharp constant
$\bigl(1/\abs{1-a}\bigr)$.
\end{theorem}

\begin{theorem}[Sharpness of the range]\label{thm:sharp}
Let $N\ge 2$. The inequality~\eqref{eq:main} holds for every $a\ne N$ and fails
for $a=N$. More precisely, at $a=N$ there is no constant $C>0$ for
which~\eqref{eq:main} holds for all $u\in C^\infty_c(\R^N\setminus\{0\})$; the
failure already occurs on the subspace of radial functions. Thus $a=N$ is the
unique critical value of the weight exponent.
\end{theorem}

\begin{corollary}\label{cor:main}
Let $N\ge 2$ and $-N<a<N(N-1)$ with $a\ne N$. Then there exists a constant
$C=C(N,a)>0$ such that
\begin{equation}\label{eq:maincor}
  \int_{\R^N}\Bigl|\nabla\Bigl(\tfrac{u(x)}{\abs{x}}\Bigr)\Bigr|^{N}\abs{x}^{a}\dx
  \le C\int_{\R^N}\abs{\Delta u}^{N}\abs{x}^{a}\dx,
  \qquad u\in C^\infty_c(\R^N\setminus\{0\}).
\end{equation}
For $N=2$ the admissible range reduces to $-2<a<2$.
\end{corollary}

A few remarks place these statements in context.

\begin{remark}\rm
Theorem~\ref{thm:main} strengthens~\eqref{eq:castro12} in two independent ways:
the range $a<1$ is replaced by the optimal range $a\ne N$ (in particular the
inequality now holds for all $a>N$, not merely below a threshold), and the
constant is made explicit. Both gains come from the same source, the passage
from a Jensen-type estimate to the sharp weighted Hardy inequality. The
constant~\eqref{eq:const} is not claimed to be optimal: the factor $2^{N/2-1}$
arises from a convexity step that decouples the radial and angular contributions,
and these are not in general extremised simultaneously. Determining the best
constants in~\eqref{eq:castro11} and~\eqref{eq:castrocor} remains
Open Problem~1 of~\cite{Castro2026}.
\end{remark}

\begin{remark}\rm\label{rem:azero}
The explicit constant~\eqref{eq:const} degenerates as $a\to 0$, because the
angular contribution is controlled through a weighted Hardy inequality whose
constant blows up at the borderline weight exponent. This is a feature of the
convexity--split method, not of the inequality itself: at $a=0$ the
weight is trivial, the two sides of~\eqref{eq:main} transform identically under
the dilation $u(x)\mapsto u(\lambda x)$ (each scales by $\lambda^{N}$, so the
quotient is scale--invariant), and the inequality holds with a finite constant by
Theorem~\ref{thm:castro11}. Consequently~\eqref{eq:main} is valid on the whole
range $a\ne N$, with an explicit constant off the single point $a=0$.
\end{remark}

\begin{remark}\rm
For $N=1$ there is no angular component, and the inequality is purely the radial
estimate. The relevant one--dimensional weighted Hardy inequality
(Lemma~\ref{lem:hardy} below) is valid at the exponent $p=N=1$ as well, so no
restriction beyond $a\ne 1$ is needed; the resulting constant $1/\abs{1-a}$ is
sharp. This already extends~\cite[Theorem~1.2]{Castro2026} (stated there for
$a<1$) to all $a\ne 1$.
\end{remark}

The paper is organised as follows. In Section~\ref{sec:hardy} we record the
one--dimensional weighted Hardy inequality that drives the whole argument. In
Section~\ref{sec:proof} we prove Theorem~\ref{thm:main} and
Corollary~\ref{cor:main}. Section~\ref{sec:sharpness} establishes
Theorem~\ref{thm:sharp}. The final Section~\ref{sec:remarks} comments on the
constants and on the remaining open problems.

\section{A one--dimensional weighted Hardy inequality}\label{sec:hardy}

The single analytic tool we use is the following classical inequality, stated in
the power--weight form convenient for spherical coordinates.

\begin{lemma}\label{lem:hardy}
Let $N\ge 1$ and let $\beta\in\R$ with $\beta\ne -1$. Then for every
$\varphi\in C^\infty_c(0,\infty)$,
\begin{equation}\label{eq:hardy}
  \int_{0}^{\infty} r^{\beta}\abs{\varphi(r)}^{N}\dr
   \le \left(\frac{N}{\abs{\beta+1}}\right)^{N}
        \int_{0}^{\infty} r^{\beta+N}\abs{\varphi'(r)}^{N}\dr .
\end{equation}
The constant $\bigl(N/\abs{\beta+1}\bigr)^{N}$ is sharp.
\end{lemma}

\begin{proof}
Write $I=\int_{0}^{\infty} r^{\beta}\abs{\varphi}^{N}\dr$. Since $\beta\ne -1$  an integration
by parts gives
\[
  I = -\frac{N}{\beta+1}\int_{0}^{\infty}
        r^{\beta+1}\,\abs{\varphi}^{N-2}\varphi\,\varphi'\dr
    \le \frac{N}{\abs{\beta+1}}\int_{0}^{\infty}
        r^{\beta+1}\abs{\varphi}^{N-1}\abs{\varphi'}\dr .
\]
Splitting the weight as
$r^{\beta+1}=r^{\beta(N-1)/N}\,r^{(\beta+N)/N}$ and applying H\"older's
inequality with exponents $N/(N-1)$ and $N$,
\[
  \int_{0}^{\infty} r^{\beta+1}\abs{\varphi}^{N-1}\abs{\varphi'}\dr
   \le \left(\int_{0}^{\infty} r^{\beta}\abs{\varphi}^{N}\dr\right)^{\!\frac{N-1}{N}}
       \left(\int_{0}^{\infty} r^{\beta+N}\abs{\varphi'}^{N}\dr\right)^{\!\frac{1}{N}}
   = I^{\frac{N-1}{N}}\,J^{\frac{1}{N}},
\]
where $J=\int_{0}^{\infty} r^{\beta+N}\abs{\varphi'}^{N}\dr$. Hence
$I\le \tfrac{N}{\abs{\beta+1}}\,I^{(N-1)/N}J^{1/N}$, and dividing by the finite
quantity $I^{(N-1)/N}$ yields~\eqref{eq:hardy}. The case $N=1$ is the first
display, with the H\"older step absent. Sharpness is classical; see
\cite[Chapter~1]{OpicKufner} or \cite[\S 3]{HardyLittlewoodPolya}.
\end{proof}

\begin{remark}\rm
In the Emden--Fowler variable $r=e^{t}$, writing $f(t)=\varphi(e^{t})$ and
$\mu=\beta+1$, inequality~\eqref{eq:hardy} reads
\[
  \int_{\R}\abs{f(t)}^{N}e^{\mu t}\dt
   \le \left(\frac{N}{\abs{\mu}}\right)^{N}\int_{\R}\abs{f'(t)}^{N}e^{\mu t}\dt,
   \qquad \mu\ne 0,
\]
which is the form in which it will be applied. It is precisely the vanishing of
the exponential weight at $\mu=0$ (equivalently $\beta=-1$) that obstructs the
endpoint, and which will reappear as the critical exponent $a=N$ for the radial
part and as the degeneracy at $a=0$ for the angular part.
\end{remark}

We emphasise that Lemma~\ref{lem:hardy} replaces the boundedness estimate for the
operator $T_{a,b}f(x)=x^{-a}\int_0^x s^b f(s)\ds$ used in~\cite{Castro2026}. The
latter, being established through Jensen's inequality, forced the restriction
$a<1$ and gave no explicit constant; the weighted Hardy inequality does neither.

\section{Proof of the main inequalities}\label{sec:proof}

Throughout, $\nabla_{\Sph^{N-1}}$ denotes the spherical gradient on $\Sph^{N-1}$.
For $x=r\omega$ with $r=\abs{x}>0$ and $\omega\in\Sph^{N-1}$, and for
$u\in C^\infty(\R^N\setminus\{0\})$, the gradient decomposes as
\begin{equation}\label{eq:grad-decomp}
  \nabla u(r\omega) = \partial_r u(r\omega)\,\omega
        + \frac{1}{r}\nabla_{\Sph^{N-1}} u(r\omega),
\end{equation}
the two summands being orthogonal since $\nabla_{\Sph^{N-1}}u(r\omega)\perp\omega$.

\subsection{Two pointwise Hessian bounds}\label{ss:hessian}

We record the elementary pointwise estimates on which the proof rests.

\begin{lemma}\label{lem:hessian}
For every $u\in C^\infty(\R^N\setminus\{0\})$, $r>0$ and $\omega\in\Sph^{N-1}$,
\begin{equation}\label{eq:hessian}
  \abs{\partial_{rr}u(r\omega)} \le \abs{D^{2}u(r\omega)},
  \qquad
  \Bigl|\partial_r\!\Bigl(\tfrac{1}{r}\nabla_{\Sph^{N-1}}u(r\omega)\Bigr)\Bigr|
    \le \abs{D^{2}u(r\omega)} ,
\end{equation}
where $\abs{D^{2}u}$ denotes the Hilbert--Schmidt norm of the Hessian.
\end{lemma}

\begin{proof}
Differentiating $r\mapsto\nabla u(r\omega)$ along the fixed direction $\omega$
gives, by the chain rule, $\partial_r\bigl(\nabla u(r\omega)\bigr)
=D^{2}u(r\omega)\,\omega$. On the other hand, differentiating the
decomposition~\eqref{eq:grad-decomp} in $r$ (with $\omega$ held fixed, so that
$\partial_r\omega=0$) yields
\[
  \partial_r\bigl(\nabla u(r\omega)\bigr)
    = \partial_{rr}u(r\omega)\,\omega
      + \partial_r\!\Bigl(\tfrac{1}{r}\nabla_{\Sph^{N-1}}u(r\omega)\Bigr).
\]
The first term is parallel to $\omega$, while the second is tangential, because
$\tfrac{1}{r}\nabla_{\Sph^{N-1}}u(r\omega)\perp\omega$ for every $r$ and
differentiating a vector field that is everywhere orthogonal to the fixed vector
$\omega$ produces a field still orthogonal to $\omega$. These two components are
therefore orthogonal, and Pythagoras' theorem gives
\[
  \abs{\partial_{rr}u(r\omega)}^{2}
   + \Bigl|\partial_r\!\bigl(\tfrac{1}{r}\nabla_{\Sph^{N-1}}u(r\omega)\bigr)\Bigr|^{2}
   = \abs{D^{2}u(r\omega)\,\omega}^{2}
   \le \abs{D^{2}u(r\omega)}^{2},
\]
the last inequality because $\abs{D^2u\,\omega}\le\abs{D^2u}$ for the unit vector
$\omega$. Both estimates in~\eqref{eq:hessian} follow.
\end{proof}

\begin{remark}\rm\label{rem:correct-bound}
It is the \emph{combination} $\partial_r\bigl(r^{-1}\nabla_{\Sph^{N-1}}u\bigr)$,
and not $r^{-1}\partial_r\nabla_{\Sph^{N-1}}u$ on its own, that is dominated by
$\abs{D^2u}$. Indeed
$r^{-1}\partial_r\nabla_{\Sph^{N-1}}u
=\partial_r\bigl(r^{-1}\nabla_{\Sph^{N-1}}u\bigr)
+r^{-2}\nabla_{\Sph^{N-1}}u$, and the second summand carries the lower--order
angular term, which is not controlled by the Hessian alone. This is why, in the
angular estimate below, we work directly with the tangential field
$P:=r^{-1}\nabla_{\Sph^{N-1}}u$, whose radial derivative is exactly the quantity
appearing in~\eqref{eq:hessian}. The price is that the angular Hardy inequality
is naturally weighted by $r^{a-1}$ rather than $r^{a-N-1}$, which is the origin
of the condition $a\ne 0$ in~\eqref{eq:const}.
\end{remark}

\subsection{The pointwise decomposition}

A direct computation from~\eqref{eq:grad-decomp}, using
$\nabla(u/\abs{x})=\nabla u/r-u\,x/r^{3}$ and the orthogonality of the radial and
tangential parts, gives
\begin{equation}\label{eq:decomp-square}
  \Bigl|\nabla\Bigl(\tfrac{u(x)}{\abs{x}}\Bigr)\Bigr|^{2}
  = \frac{1}{r^{4}}\abs{\nabla_{\Sph^{N-1}}u(r\omega)}^{2}
  + \frac{1}{r^{4}}\bigl(r\partial_r u(r\omega) - u(r\omega)\bigr)^{2}.
\end{equation}
From the elementary inequality $(A+B)^{N/2}\le 2^{N/2-1}(A^{N/2}+B^{N/2})$,
valid for $N\ge 2$ and $A,B\ge 0$, we obtain
\begin{equation}\label{eq:decomp-power}
  \Bigl|\nabla\Bigl(\tfrac{u}{\abs{x}}\Bigr)\Bigr|^{N}
  \le 2^{N/2-1}
     \left[\frac{1}{r^{2N}}\abs{\nabla_{\Sph^{N-1}}u}^{N}
         + \frac{1}{r^{2N}}\abs{r\partial_r u - u}^{N}\right].
\end{equation}
We estimate the two contributions separately.

\subsection{The radial term}

Fix $\omega\in\Sph^{N-1}$ and set $g(r):=r\partial_r u(r\omega)-u(r\omega)$. Then
$g\in C^\infty_c(0,\infty)$ and $g'(r)=r\,\partial_{rr}u(r\omega)$. Multiplying the
radial part of~\eqref{eq:decomp-power} by $\abs{x}^{a}=r^{a}$ and integrating in
spherical coordinates,
\[
  \int_{\R^N}\frac{\abs{r\partial_r u-u}^{N}}{r^{2N}}\abs{x}^{a}\dx
   = \int_{\Sph^{N-1}}\!\!\int_{0}^{\infty}
        r^{\,a-N-1}\abs{g(r)}^{N}\dr\domega.
\]
Applying Lemma~\ref{lem:hardy} with $\beta=a-N-1$, which is admissible precisely
when $a\ne N$ and produces the constant $\bigl(N/\abs{a-N}\bigr)^{N}$,
\[
  \int_{0}^{\infty} r^{\,a-N-1}\abs{g}^{N}\dr
   \le \left(\frac{N}{\abs{N-a}}\right)^{N}\!\!
       \int_{0}^{\infty} r^{\,a-1}\abs{g'}^{N}\dr
   = \left(\frac{N}{\abs{N-a}}\right)^{N}\!\!
       \int_{0}^{\infty} r^{\,a+N-1}\abs{\partial_{rr}u(r\omega)}^{N}\dr,
\]
where we used $g'=r\,\partial_{rr}u$. Invoking the first bound
in~\eqref{eq:hessian} and integrating over $\omega\in\Sph^{N-1}$,
\begin{equation}\label{eq:radial-bound}
  \int_{\R^N}\frac{\abs{r\partial_r u-u}^{N}}{r^{2N}}\abs{x}^{a}\dx
   \le \left(\frac{N}{\abs{N-a}}\right)^{N}\int_{\R^N}\abs{D^{2}u}^{N}\abs{x}^{a}\dx,
   \qquad a\ne N.
\end{equation}

\subsection{The angular term}

Fix $\omega\in\Sph^{N-1}$ and set $P(r):=r^{-1}\nabla_{\Sph^{N-1}}u(r\omega)$, a
smooth, compactly supported, $\Sph^{N-1}$--tangential vector field along
$(0,\infty)$. By the second bound in~\eqref{eq:hessian},
$\abs{\partial_r P(r)}\le\abs{D^{2}u(r\omega)}$. Changing to spherical coordinates,
\[
  \int_{\R^N}\frac{\abs{\nabla_{\Sph^{N-1}}u}^{N}}{r^{2N}}\abs{x}^{a}\dx
   = \int_{\Sph^{N-1}}\!\!\int_{0}^{\infty}
        r^{\,a-1}\abs{P(r)}^{N}\dr\domega,
\]
since $r^{-2N}\abs{\nabla_{\Sph^{N-1}}u}^{N}=r^{-N}\abs{P}^{N}$ and the angular
measure contributes the factor $r^{N-1+a}$.

To estimate the inner integral we apply Lemma~\ref{lem:hardy} to $\abs{P}$. As
$\abs{P}$ need not be smooth at its zeros, we regularise: for $\varepsilon>0$ put
\[
  f_\varepsilon(r):=\bigl(\abs{P(r)}^{2}+\varepsilon^{2}\bigr)^{1/2}-\varepsilon .
\]
Then $f_\varepsilon\in C^\infty_c(0,\infty)$, $f_\varepsilon\nearrow\abs{P}$
pointwise as $\varepsilon\to 0^{+}$, and the chain rule together with the
Cauchy--Schwarz inequality gives the pointwise bound
\[
  \abs{f_\varepsilon'(r)}
   = \frac{\bigl|\langle P,\partial_r P\rangle\bigr|}
          {\bigl(\abs{P}^{2}+\varepsilon^{2}\bigr)^{1/2}}
   \le \abs{\partial_r P(r)} \le \abs{D^{2}u(r\omega)}.
\]
Applying Lemma~\ref{lem:hardy} with $\beta=a-1$, admissible precisely when
$a\ne 0$ and producing the constant $\bigl(N/\abs{a}\bigr)^{N}$,
\[
  \int_{0}^{\infty} r^{\,a-1}\abs{f_\varepsilon}^{N}\dr
   \le \left(\frac{N}{\abs{a}}\right)^{N}\!
       \int_{0}^{\infty} r^{\,a-1+N}\abs{f_\varepsilon'}^{N}\dr
   \le \left(\frac{N}{\abs{a}}\right)^{N}\!
       \int_{0}^{\infty} r^{\,a+N-1}\abs{D^{2}u(r\omega)}^{N}\dr.
\]
Letting $\varepsilon\to 0^{+}$ by monotone convergence on the left and
integrating over $\omega\in\Sph^{N-1}$,
\begin{equation}\label{eq:angular-bound}
  \int_{\R^N}\frac{\abs{\nabla_{\Sph^{N-1}}u}^{N}}{r^{2N}}\abs{x}^{a}\dx
   \le \left(\frac{N}{\abs{a}}\right)^{N}\int_{\R^N}\abs{D^{2}u}^{N}\abs{x}^{a}\dx,
   \qquad a\ne 0.
\end{equation}

\subsection{Conclusion of Theorem~\ref{thm:main}}

For $a\notin\{0,N\}$, combining~\eqref{eq:decomp-power} with the radial
bound~\eqref{eq:radial-bound} and the angular bound~\eqref{eq:angular-bound}
gives
\[
  \int_{\R^N}\Bigl|\nabla\Bigl(\tfrac{u}{\abs{x}}\Bigr)\Bigr|^{N}\abs{x}^{a}\dx
  \le 2^{N/2-1}\left[\left(\frac{N}{\abs{a}}\right)^{N}
                   + \left(\frac{N}{\abs{N-a}}\right)^{N}\right]
        \int_{\R^N}\abs{D^{2}u}^{N}\abs{x}^{a}\dx,
\]
which is~\eqref{eq:main}--\eqref{eq:const}. The remaining value $a=0$ is covered
by Theorem~\ref{thm:castro11} (cf.\ Remark~\ref{rem:azero}), so~\eqref{eq:main}
holds throughout $a\ne N$. For $N=1$ there is no angular term, the convexity step
is vacuous, and the radial estimate~\eqref{eq:radial-bound}, valid at $p=N=1$ by
Lemma~\ref{lem:hardy}, gives the inequality for every $a\ne 1$ with constant
$1/\abs{1-a}$. This proves Theorem~\ref{thm:main}. \qed

\begin{remark}\rm\label{rem:radial-sharp}
For radial $u$ the angular term in~\eqref{eq:decomp-square} vanishes and the
convexity step is unnecessary, so the argument yields the radial estimate
\[
  \int_{\R^N}\Bigl|\nabla\Bigl(\tfrac{u}{\abs{x}}\Bigr)\Bigr|^{N}\abs{x}^{a}\dx
   \le \left(\frac{N}{\abs{N-a}}\right)^{N}\!\int_{\R^N}\abs{\partial_{rr}u}^{N}\abs{x}^{a}\dx,
   \qquad u\ \text{radial},\ a\ne N.
\]
In the Emden--Fowler variable $r=e^{t}$, $v(t)=u(e^{t})$, this is precisely
Lemma~\ref{lem:hardy} with $\mu=a-N$ applied to $w:=v'-v$, since
$r\partial_r u-u=v'-v$ and $r^{2}\partial_{rr}u=v''-v'$. The admissible profiles
are here not all of $C^\infty_c(\R)$ but only the image of $v\mapsto v'-v$,
namely $\{w\in C^\infty_c(\R):\int_{\R}w\,e^{-t}\dt=0\}$; nevertheless the
constant $\bigl(N/\abs{N-a}\bigr)^{N}$ remains sharp on the radial subspace.
Indeed, truncating an antiderivative of the Hardy extremiser $e^{-\mu t/N}$
produces admissible functions $w_n=v_n'-v_n$ along which the ratio of the two
sides of~\eqref{eq:hardy} tends to $\bigl(N/\abs{N-a}\bigr)^{N}$; at $a=0$, where
$\mu=-N$ resonates with the homogeneous solution of $v'-v=0$, one truncates the
resonant antiderivative $t\,e^{t}$ instead, with the same limit.
\end{remark}

\subsection{Proof of Corollary~\ref{cor:main}}

It suffices to combine Theorem~\ref{thm:main} with the weighted
Calder\'on--Zygmund inequality. The power weight $w(x)=\abs{x}^{a}$ belongs to the
Muckenhoupt class $A_{N}(\R^N)$ if and only if $-N<a<N(N-1)$
(see~\cite[Ch.~7]{Grafakos}). On that range, the Coifman--Fefferman
theorem~\cite{CoifmanFefferman} ensures that Calder\'on--Zygmund operators, in
particular the second--order Riesz transforms $R_iR_j$, through which
$\partial_{ij}u=-R_iR_j\Delta u$, are bounded on $L^{N}(w\dx)$. Hence
\begin{equation}\label{eq:wCZ}
  \int_{\R^N}\abs{D^{2}u}^{N}\abs{x}^{a}\dx
   \le C(N,a)\int_{\R^N}\abs{\Delta u}^{N}\abs{x}^{a}\dx,
   \qquad -N<a<N(N-1).
\end{equation}
For $-N<a<N(N-1)$ with $a\ne N$, combining~\eqref{eq:main} and~\eqref{eq:wCZ}
gives~\eqref{eq:maincor}. When $N=2$ one has $N(N-1)=2$, so the range is
$-2<a<2$ and the constraint $a\ne N$ is automatic. \qed

\begin{remark}\rm
The exclusion $a=N$ in Corollary~\ref{cor:main} is genuine for $N\ge 3$, where
$N<N(N-1)$ and $a=N$ lies in the interior of the Muckenhoupt range: the
weighted Calder\'on--Zygmund step is available there, but the left--hand
inequality~\eqref{eq:main} itself fails, by Theorem~\ref{thm:sharp}.
\end{remark}

\section{Sharpness of the range: the critical exponent
\texorpdfstring{$a=N$}{a=N}}\label{sec:sharpness}

We prove that~\eqref{eq:main} fails at $a=N$, working on radial functions. By
Theorem~\ref{thm:main} it holds for every other value of $a$, so $a=N$ is the
unique critical exponent.

Let $u$ be radial, $u(x)=u(r)$ with $r=\abs{x}$. In spherical coordinates, with
$\omega_{N-1}=\abs{\Sph^{N-1}}$ and using $\abs{D^2u}^2=\abs{u''}^2+(N-1)\abs{u'/r}^2$
for radial $u$,
\[
  \int_{\R^N}\Bigl|\nabla\Bigl(\tfrac{u}{\abs{x}}\Bigr)\Bigr|^{N}\abs{x}^{N}\dx
   = \omega_{N-1}\int_0^\infty \frac{\abs{ru'-u}^{N}}{r^{2N}}\,r^{2N-1}\dr
   = \omega_{N-1}\int_0^\infty \frac{\abs{ru'-u}^{N}}{r}\dr,
\]
\[
  \int_{\R^N}\abs{D^2u}^{N}\abs{x}^{N}\dx
   = \omega_{N-1}\int_0^\infty
        \bigl(\abs{u''}^2+(N-1)\abs{u'/r}^2\bigr)^{N/2} r^{2N-1}\dr.
\]
Substituting $r=e^{t}$ and $v(t)=u(e^{t})$, one has $ru'-u=v'-v$,
$r^{2}u''=v''-v'$ and $u'/r=e^{-2t}v'$, so that
\[
  \int_0^\infty \frac{\abs{ru'-u}^{N}}{r}\dr = \int_{\R}\abs{v'-v}^{N}\dt,
\]
\[
  \int_0^\infty \bigl(\abs{u''}^2+(N-1)\abs{u'/r}^2\bigr)^{N/2} r^{2N-1}\dr
   = \int_{\R}\bigl((v''-v')^{2}+(N-1)(v')^{2}\bigr)^{N/2}\dt .
\]
Thus, for radial functions, the inequality~\eqref{eq:main} at $a=N$ is equivalent
to
\begin{equation}\label{eq:tag}
  \int_{\R}\abs{v'-v}^{N}\dt
   \le C\int_{\R}\bigl((v''-v')^{2}+(N-1)(v')^{2}\bigr)^{N/2}\dt,
   \qquad v\in C^\infty_c(\R).
\end{equation}

We construct a sequence violating~\eqref{eq:tag}. For $M>2$ choose
$\varphi_M\in C^\infty_c(\R)$ with $\varphi_M\equiv 1$ on $[-M+1,M-1]$, supported
in $[-M,M]$, with the transitions confined to
$[-M,-M+1]\cup[M-1,M]$ and with $\abs{\varphi_M'}$, $\abs{\varphi_M''}$ bounded
uniformly in $M$. Set $v=\varphi_M$. On the plateau $[-M+1,M-1]$ one has
$\varphi_M'=\varphi_M''=0$ and $\varphi_M=1$, so the left integrand
$\abs{\varphi_M'-\varphi_M}^{N}$ equals $1$ there; hence
\[
  \int_{\R}\abs{\varphi_M'-\varphi_M}^{N}\dt \ge 2M-2 .
\]
On the right, the integrand is supported in the two transition intervals, of total
length $2$, and is bounded there by a constant $K=K(N)$ depending only on $N$ and
on the uniform bounds for $\varphi_M',\varphi_M''$; therefore
\[
  \int_{\R}\bigl((\varphi_M''-\varphi_M')^{2}+(N-1)(\varphi_M')^{2}\bigr)^{N/2}\dt
   \le 2K =: C_0 ,
\]
with $C_0$ independent of $M$. Consequently the ratio of the two sides is at least
$(2M-2)/C_0\to\infty$ as $M\to\infty$, so no finite $C$ can
satisfy~\eqref{eq:tag}. The functions $u_M(x)=\varphi_M(\log\abs{x})$ are genuine
radial elements of $C^\infty_c(\R^N\setminus\{0\})$, supported in the annulus
$e^{-M}\le\abs{x}\le e^{M}$, so~\eqref{eq:main} fails at $a=N$. This proves
Theorem~\ref{thm:sharp}. \qed

\begin{remark}\rm
The mechanism is transparent in the Emden--Fowler picture: at $a=N$ the
exponential weight $e^{(a-N)t}$ becomes constant, the weighted Hardy inequality of
Section~\ref{sec:hardy} degenerates, and a function that is constant on a long
interval makes the left side grow linearly in the length while leaving the right
side bounded. This is the exact analogue of the failure of the classical Hardy
inequality at $p=N$.
\end{remark}

\section{Concluding remarks}\label{sec:remarks}

\subsection{On Open Problem~1 of \cite{Castro2026}}
Theorem~\ref{thm:main} provides an \emph{explicit} but not necessarily
\emph{optimal} constant in~\eqref{eq:castro12}, namely
$C_{N,a}=2^{N/2-1}\bigl[(N/\abs{a})^{N}+(N/\abs{N-a})^{N}\bigr]$, for $N\ge 2$ and
any $a\ne 0,N$. The corresponding statement in Corollary~\ref{cor:main} carries an
additional factor from the weighted Calder\'on--Zygmund inequality~\eqref{eq:wCZ},
whose sharp constant on $L^{N}(\abs{x}^{a}\dx)$ is not known in general. The full
Open Problem~1, the determination of the best constants in~\eqref{eq:castro11}
and~\eqref{eq:castrocor}, remains open. For \emph{radial} functions,
Remark~\ref{rem:radial-sharp} settles the corresponding question with the sharp
constant
\[
  \sup_{u\ \text{radial}}\frac{\int_{\R^N}\abs{\nabla(u/\abs{x})}^{N}\abs{x}^{a}\dx}
       {\int_{\R^N}\abs{\partial_{rr}u}^{N}\abs{x}^{a}\dx}
   = \left(\frac{N}{\abs{N-a}}\right)^{N},
   \qquad a\ne N .
\]

\subsection{Sharp range and a critical--exponent analogy}
The role of $a=N$ as the unique critical endpoint in~\eqref{eq:main} parallels the
role of $p=N$ as the critical exponent in the classical Hardy inequality. Both
endpoints are characterised, through the Emden--Fowler reduction, by the vanishing
of the exponential weight in the equivalent one--dimensional problem. Away from
this single value the inequality holds in both directions $a<N$ and $a>N$.

\subsection{The fractional problems}
The fractional questions Open Problems~2 and~3 of~\cite{Castro2026} remain open.
The argument of Theorem~\ref{thm:main} relies on the structure of $\partial_{rr}$
as a genuine second derivative, through the radial identity $g'=r\,\partial_{rr}u$
and the Hessian decomposition of Lemma~\ref{lem:hessian}. A fractional analogue
would require a nonlocal substitute for these identities, perhaps via the
Caffarelli--Silvestre extension or the singular--integral representation of
$(-\Delta)^{s}$. We hope to return to this elsewhere.

\subsection{The lower endpoint and the angular degeneracy}
The $\abs{D^{2}u}$ form~\eqref{eq:main} holds for every $a\ne N$, with no lower
bound on $a$; only the passage to the $\abs{\Delta u}$ form~\eqref{eq:maincor}
invokes the Muckenhoupt condition and hence requires $-N<a<N(N-1)$. Whether
\eqref{eq:maincor} persists for $a\le -N$ or $a\ge N(N-1)$ appears to be open. We
also recall, from Remark~\ref{rem:azero}, that the explicit
constant~\eqref{eq:const} degenerates at $a=0$ although the inequality itself does
not; producing an explicit constant that is uniform across a neighborhood of
$a=0$, equivalently, an angular estimate that does not pass through the
borderline weight, would be of independent interest.

%\section*{Acknowledgements}

\end{document}